\documentclass{amsart}

\usepackage{amssymb}
\usepackage{amsthm}
\usepackage{amsmath}
\usepackage{amsbsy}
\usepackage[all]{xy}
\usepackage{bm}
\usepackage{hyperref}
\usepackage{tikz}
\newtheorem{thm}{Theorem}
\newtheorem*{lem}{Lemma}
\newtheorem*{proposition}{Proposition}
\newtheorem{corollary}{Corollary}
\newtheorem{lemma}{Lemma}

\title[]{Exponential Sums and Riesz Energies}
\author[]{Stefan Steinerberger}
\address{Department of Mathematics, Yale University}
\email{stefan.steinerberger@yale.edu}

\begin{document}
\begin{abstract} We bound an exponential sum that appears in the study of irregularities of distribution (the low-frequency Fourier energy of the sum of several Dirac measures) by geometric quantities: a special case is that for all $\left\{ x_1, \dots, x_N\right\} \subset \mathbb{T}^2$, $X \geq 1$ and a universal $c>0$
 $$   \sum_{i,j=1}^{N}{ \frac{X^2}{1 +  X^4 \|x_i -x_j\|^4}} \lesssim \sum_{k \in \mathbb{Z}^2 \atop \|k\| \leq X}{ \left| \sum_{n=1}^{N}{ e^{2 \pi i \left\langle k, x_n \right\rangle}}\right|^2} \lesssim  \sum_{i,j=1}^{N}{ X^2 e^{-c X^2\|x_i -x_j\|^2}}.$$
Since this exponential sum is intimately tied to rather subtle distribution properties of the points, we obtain nonlocal structural statements for near-minimizers of the Riesz-type energy. For $X \gtrsim N^{1/2}$ both upper and lower bound match for maximally-separated point sets satisfying $\|x_i -x_j\| \gtrsim N^{-1/2}$. 
\end{abstract}

\maketitle

\section{Introduction and Main Results}
\subsection{Introduction}
Let $\left\{ x_1, \dots, x_N\right\} \subset \mathbb{T}^2$ (throughout this paper normalized to $\mathbb{T}^2 \cong [0,1]^2$). Montgomery's theorem \cite{mont1} (see also Beck \cite{beck1, beck}) is a classical example of an irregularity of distribution phenomenon: there exists a disk $D \subset \mathbb{T}^2$ with radius $1/4$
or $1/2$ such that the number of elements in the disk substantially deviates from its expectation
$$ \left| \# \left\{1 \leq i \leq N: x_i \in D\right\} - N|D|\right| \gtrsim N^{1/4}.$$
One way to prove this type of result is by estimating the $L^{\infty}-$norm from
below by the $L^2-$norm and then use Parseval's identity to separate the Fourier transform
of the characteristic function of the geometric shape (here: a disk) and the Fourier transform of Dirac measures located at $\left\{ x_1, \dots, x_N\right\} \subset \mathbb{T}^2$
$$ \widehat{ \sum_{n=1}^{N}{\delta_{x_n}}} = \sum_{k \in \mathbb{Z}^2}{ \left( \sum_{n=1}^{N}{ e^{-2 \pi i \left\langle k, x_n \right\rangle}}\right) e^{2\pi i \left\langle k, x\right\rangle} }.$$
A fundamental ingredient of the method is the fact that the Fourier transform of Dirac measures cannot be too small on low frequencies.

\begin{lem}[Montgomery \cite{mont1}] For any $\left\{ x_1, \dots, x_N\right\} \subset \mathbb{T}^2$ and $X_1, X_2 \geq 0$
 $$\sum_{|k_1| \leq X_1 \atop |k_2| \leq X_2}{ \left| \sum_{n=1}^{N}{ e^{2 \pi i \left\langle k, x_n \right\rangle}}\right|^2} \geq N X_1 X_2.$$
\end{lem}
This inequality may be interpreted as a two-dimensional analogue of a result of Cassels \cite{cassels} and is related to a result of Siegel \cite{siegel}. It is essentially sharp: let $p$ be a prime, $x_n = (n/p, 0)$ for $0 \leq n \leq p-1$ and set $X_1 = p-1$ and $X_2$ = 1, then
  $$3 p^2 = \sum_{|k_1| \leq p-1 \atop |k_2| \leq 1}{ \left| \sum_{n=1}^{p}{ e^{2 \pi i \left\langle k, x_n \right\rangle}}\right|^2} \geq p(p-1).$$
A corollary below will show that the inequality is sharp for all $X_1 = X_2 \gtrsim N^{1/2}$ whenever the points are maximally separated (meaning $\|x_i -x_j\| \gtrsim N^{-1/2}$ for  $i \neq j$). The interesting parameter range for the inequality to be nontrivial is $X_1 X_2 \gtrsim N$ because the term for $(k_1, k_2) = (0,0)$ has size $N^2$.

\subsection{The result.} The purpose of this paper is to point out that there is a lower bound that connects a nonlocal functional resembling a Riesz energy to the exponential sum; for simplicity, we first state and discuss the case $X_1=X_2$ in two dimensions and refer to \S 1.4. for the general formulation.

\begin{thm}[Special case $X_1=X_2$, $d=2$] For all $\left\{ x_1, \dots, x_N\right\} \subset \mathbb{T}^2$ and $X \geq 0$
 $$\sum_{\|k\| \leq X}{ \left| \sum_{n=1}^{N}{ e^{2 \pi i \left\langle k, x_n \right\rangle}}\right|^2} \gtrsim \sum_{i,j=1}^{N}{ \frac{X^2}{1 +  X^4 \|x_i -x_j\|^4}}.$$
\end{thm}
 By only summing over the diagonal terms $i=j$, we recover the original estimate $NX^2$ up to constants, however,
 the off-diagonal terms carry additional information: if there are many pairs of points with $\| x_i - x_j\| \lesssim X^{-1}$,
 they can contribute substantially. This allows us to slightly refine existing results where this exponential sum plays a role (see Corollary 1 or \cite{mont} for more examples).
\begin{corollary}
Let $\left\{ x_1, \dots, x_N\right\} \subset \mathbb{T}^2$. There exists a disk $D$ of radius $1/4$
or $1/2$ such that the number of elements in the disk deviates from what its area predicts by 
$$ \left| \# \left\{1 \leq i \leq N: x_i \in D\right\} - N|D|\right| \gtrsim \left( N^{-3/2} \sum_{i,j=1}^{N}{ \frac{N}{1 +  N^2 \|x_i -x_j\|^4}}  \right)^{1/2}.$$
\end{corollary}
Note that this Riesz-type energy is always $\gtrsim N^{2}$ which recovers the lower bound $\gtrsim N^{1/4}$. The statement implies that point sets for which the energy is $\gg N^2$ have to have a subtle form of clustering where large disks $D$ of radius $1/4$ or $1/2$ deviate `substantially' from having $|D|N$ points. We emphasize that `substantially' is on the scale $\sim N^{\alpha}$, where $\alpha$ can be as low as $1/4$ -- the results operate on extremely fine scales.
Corollary 1 follows from Montgomery's original proof (see \cite{mont1, mont}) and then applying Theorem 1 in the last step. To the best of our knowledge, this is the first such bound for exponential sums and the first such statement for near-minimizers of a Riesz-type energy.  The connection between discrepancy and Riesz energy has been investigated before: we refer to Brauchart \cite{brauchart}, Bilyk \& Dai \cite{bil2}, Bilyk, Dai \& Matzke \cite{bil3}, Bilyk \& Lacey \cite{bil1} and Leopardi \cite{leopardi}.
 
 \subsection{Riesz energies.} This relates to existing results (see e.g. \cite{saff2}) regarding the behavior of
  $N$ points on a given manifold interacting via quantities like the Riesz kernel $\|x-y\|^{-s}$. The question is how to minimize, over all sets of $N$ points,
    $$ \sum_{i \neq j}{ f(\|x_i -x_j\|)} \qquad \mbox{and to determine the structure of minimizers.}$$
  
 These problems have a long history: the special case $\|x_i - x_j\|^{-1}$ on $\mathbb{S}^2$ is usually interpreted as the minimal energy configuration of $N$ electrons on a sphere and dates back to the physicist J. J. Thomson \cite{thomson} in 1904 (see Schwartz \cite{schwartz} for the recent solution of the case $N=5$). The case $f(\|x_i-x_j\|) = \|x_i - x_j\|^{-s}$ is usually refered to as Riesz energy and arises in many different settings (we refer to the surveys \cite{ saffneu1, saff2}, to \cite{saffneu3, saffneu2} for more recent results, to \cite{stein} for an application in combinatorial geometry and to the survey \cite{blanc} for the larger family of problems surrounding crystallization). A classical result for $\left\{x_1, \dots, x_N\right\} \subset \mathbb{T}^2$ (and other compact two-dimensional manifolds) is
 $$  \sum_{i,j=1 \atop i \neq j}^{N}{ \frac{1}{ \|x_i -x_j\|^2}} \gtrsim N^2 \log{N} $$
 and is well understood (see Kuijlaars \& Saff \cite{saff} for a much more precise result on the sphere). Minimizers are `roughly' evenly spaced and what remains to be understood are fine structural details of the minimizing configuration. An old question is under which circumstances the spacings in a minimizing configuration uniformly satisfy the optimal lower bound $\|x_i - x_j\| \gtrsim N^{-1/d}$  (see Dahlberg \cite{dahlberg} for one of the first results on this and Hardin, Reznikov, Saff \& Volberg \cite{vol} for a more recent result). This question is related to whether the singularity contributes substantially to the energy of a minimizing configuration; motivated by the Riesz-type energy appearing in Theorem 1, we prove that this is not the case.
 
 \begin{thm} For all $\left\{x_1, \dots, x_N\right\} \subset \mathbb{T}^2$
  $$  \sum_{i,j=1 \atop i \neq j}^{N}{ \frac{1}{1 +N \|x_i -x_j\|^2}} \gtrsim N \log{N}.$$
 \end{thm}
 
 The proof is not particularly two-dimensional in flavor; we also do not make any special use of the structure of $\mathbb{T}^2$ and various generalizations suggest themselves but are beyond the scope of this paper. Since this kernel is strictly smaller than $N^{-1}\|x_i- x_j\|^{-2}$  and without a singularity, Theorem 2 strenghtens the classical lower bound for the Riesz energy. A natural question is whether optimal configurations $\left\{x_1, \dots, x_N\right\} \subset \mathbb{T}^2$ minimizing the energy in Theorem 2 are also well-separated and satisfy $\|x_i - x_j\| \gtrsim N^{-1/2}$. This is probably harder than in the Riesz case.\\
 
 The main point can be summarized as follows:
 a fundamental question is how minimizers or near-minimizers of Riesz energies behave; Theorem 1 implies that whenever the Riesz-type energy
 $$  \sum_{i,j=1}^{N}{ \frac{1}{1 +  N^2 \|x_i -x_j\|^4}} \quad \mbox{is slightly larger than}~N \log{N},$$
 then this implies the presence of nontrivial and avoidable global irregularities of distribution phenomena (for instance, the existence of disks containing significantly more or less points than their area suggests to a larger extent than for other sets of points with the same cardinality). We believe this to be fairly powerful indicator that relatively simple Riesz-energy type quantities can capture somewhat subtle combinatorial properties of the set. Another simple example is obtained from reversing the direction of the statement: if $\left\{x_1, \dots, x_N \right\} \subset \mathbb{T}^2$ is a set of points such that every disk $D \subset \mathbb{T}^2$ of radius $1/4$ or $1/2$ contains $|D|N \pm c N^{1/4} \sqrt{\log{N}}$ points, then
 $$  \sum_{i,j=1}^{N}{ \frac{1}{1 +  N^2 \|x_i -x_j\|^4}} \lesssim_c N\log{N}.$$
 This seems like a statement that would be fairly difficult to prove via other means (because having a surplus of $\sim N^{1/4} \sqrt{\log{N}}$ points in a disk of area $\sim 1$ does not seem enough for the usual convexity arguments to yield an improvement).
 We believe that these results raise a large number of natural question and hope that they will inspire
 further work on the intersection of these fields.

\subsection{Lower and Upper Bounds.} We now describe the lower and upper bound.
\begin{thm} For all $\left\{ x_1, \dots, x_N\right\} \subset \mathbb{T}^d$ and all $X_1, X_2, \dots, X_d >0$
 $$\sum_{k \in \mathbb{Z}^d \atop |k_m| \leq X_m}{ \left| \sum_{n=1}^{N}{ e^{2 \pi i \left\langle k, x_n \right\rangle}}\right|^2} \gtrsim_d
 \sum_{i,j=1}^{N}{ \prod_{m=1}^{d}{ \frac{X_m \log{(e + N|  (x_{i,m} -x_{j,m})|)}}{1 +  X_m^2 (x_{i,m} -x_{j,m})^2} }}.$$
\end{thm}
$x_{i,m}$ denotes the $m-$th coordinate of the point $x_i$. We note that the special case in Theorem 1 does
not contain the logarithm (for simplicity of exposition and because it does not have a major impact on any of the asymptotics).
As for the upper bound, we show the following.
\begin{thm} There exists $c_d>0$ such that for all $\left\{x_1, \dots, x_N\right\} \subset \mathbb{T}^d$ and $X \geq 1$
 $$\sum_{k \in \mathbb{Z}^d \atop \|k\| \leq X}{ \left| \sum_{n=1}^{N}{ e^{2 \pi i \left\langle k, x_n \right\rangle}}\right|^2} \lesssim_d   \sum_{i,j=1}^{N}{ X^d e^{-c_d X^2\|x_i -x_j\|^2}}.$$
\end{thm} 
The extension to the more general case $|k_i| \leq X_i$ can be derived as well, we quickly remark after the proof how this could be done.
Theorem 4 can be used to show that in the regime $X \gtrsim N^{1/d}$, when Montgomery's Lemma starts being effective, it is also rather precise and sharp for well-separated point sets.
\begin{corollary} If the points $\left\{x_1, \dots, x_N\right\} \subset \mathbb{T}^d$ satisfy $\|x_i - x_j\| \gtrsim N^{-1/d}$ for all $i \neq j$, then for $X \gtrsim N^{1/d}$ we have matching upper and lower bounds
 $$\sum_{k \in \mathbb{Z}^d \atop \|k\| \leq X}{ \left| \sum_{n=1}^{N}{ e^{2 \pi i \left\langle k, x_n \right\rangle}}\right|^2} \sim NX^d.$$
\end{corollary} 
 Corollary 2 shows that the regime $X \gtrsim N^{1/d}$ is thus more or less understood; we believe that it could be of interest to study the precise behavior in the regime $X \lesssim N^{1/d}$ without the dominating term $N^2$ at $k=0$, i.e. to understand
$$\sum_{ \|k\| \leq X \atop k \neq 0}{ \left| \sum_{n=1}^{N}{ e^{2 \pi i \left\langle k, x_n \right\rangle}}\right|^2} \qquad \mbox{for}~X \lesssim N^{1/d}.$$
Generally, it is not going to be possible to obtain nontrivial results: if $p \leq N^{1/d}$ is prime, then we may consider the set
of $p^d$ points
$$ \left\{\left( \frac{i_1}{p}, \frac{i_2}{p}, \dots, \frac{i_d}{p}\right): 0 \leq i_1, \dots, i_d \leq p-1\right\}$$
and observe that for $k \neq 0$ and $\|k\|_{\ell^{\infty}} \leq p-1$
$$  \sum_{n=1}^{N}{ e^{2 \pi i \left\langle k, x_n \right\rangle}}  = 0.$$
However, it remains entirely unclear what happens if $N \neq p^d$ for $p$ prime.  We observe that there exist $\left\{x_1, \dots, x_N \right\} \subset \mathbb{T}^d$ such that for all $X \geq 1$
$$ \sum_{ \|k\| \leq X \atop k \neq 0}{ \left| \sum_{n=1}^{N}{ e^{2 \pi i \left\langle k, x_n \right\rangle}}\right|^2} \lesssim X^{d+2} (\log{N})^{2d-2}.$$
For $X \sim N^{1/d}$, this coincides with the optimal result $\sim NX^2$ up to a logarithm.
If  $\left\{x_1, \dots, x_N \right\} \subset \mathbb{T}^d$ is a low-discrepancy set and satisfies
$$ \mbox{extreme Discrepancy}\left(\left\{x_1, \dots, x_N \right\}\right) \lesssim \frac{\left(\log{N}\right)^{d-1}}{N},$$
where we refer to \cite{pill,kuip} for the relevant definitions and constructions of such sets, then the Koksma-Hlawka inequality implies
$$ \left| \sum_{n=1}^{N}{ e^{2 \pi i \left\langle k, x_n \right\rangle}}\right| =  \left| \sum_{n=1}^{N}{ e^{2 \pi i \left\langle k, x_n \right\rangle}} - N\int_{\mathbb{T}^2}{e^{2\pi i \left\langle k, x\right\rangle } dx}\right| \lesssim \mbox{var}(e^{2\pi i \left\langle k, x\right\rangle }) \left(\log{N}\right)^{d-1},$$
where $\mbox{var}$ refers to the total variation. Since $\mbox{var}(e^{2\pi i \left\langle k, x\right\rangle }) \sim \|k\|$, we obtain
$$ \sum_{ \|k\| \leq X \atop k \neq 0}{ \left| \sum_{n=1}^{N}{ e^{2 \pi i \left\langle k, x_n \right\rangle}}\right|^2} \lesssim
\sum_{ \|k\| \leq X}{\|k\|^2 (\log{N})^{2d-2}} \sim X^{d+2}  (\log{N})^{2d-2}.$$
The next section discusses a connection between this question and a problem in irregularities of distribution that could be of independent interest. 

\subsection{Irregularities of (Heat) Distribution.} Suppose $f \in L^1(\mathbb{T}^d)$ is given and $f(x)$ describes the temperature in the point $x$. We are allowed to take $N$ measurements and would like to know the average temperature in the room, i.e.
estimate
$$ \int_{\mathbb{T}^d}{f(x) dx}.$$
However, the new ingredient is that we do not necessarily need to know the answer right away, it suffices to know the answer within $t$ units of time. In this time period, the temperature evolves according to the heat equation. We use $e^{t\Delta}f(x)$ to
denote the temperature at time $t$. Since the heat equation is smoothing and preserves the average value, it makes sense to wait until all $t$ units of time have passed and then sample $e^{t\Delta}f$ in $\left\{x_1, \dots, x_N\right\} \in \mathbb{T}^d$.

\begin{proposition} We have for all $f \in L^1(\mathbb{T}^d)$
$$ \sup_{\|f\|_{L^1(\mathbb{T}^d)} \leq 1} \left| \frac{1}{N} \sum_{i=1}^{N}{e^{t\Delta}f(x_i)} - \int_{\mathbb{T}^d}{f(x) dx}  \right| = \frac{1}{N} \left\| 
N -  \sum_{i=1}^{N}{e^{t\Delta}\delta_{x_i}} \right\|_{L^{\infty}(\mathbb{T}^d)}.$$
\end{proposition}

This shows that the problem is equivalent to the question of where to place $N$ Dirac measures (with weight $1/N$) so that their heat
evolution after $t$ units of time is as close to a constant function as possible.
\begin{quote}
\textbf{Open problem.} Given $t$ and $N$, how small can the error be and how would one arrange 
$\left\{x_1, \dots, x_N\right\} \in \mathbb{T}^d$ to achieve it?
\end{quote}
The interesting regime is $t \gtrsim 1/N^{2/d}$: we recall that $e^{t\Delta} \delta_x$ behaves essentially like a Gaussian supported on scale $\sim \sqrt{t}$. As soon as $t \gtrsim 1/N^{2/d}$ the Gaussians are so wide that they start to overlap and it becomes possible to obtain better and better approximations of the constant function. The next statement shows that trying to understand optimal sampling schemes is intimately tied to the behavior of the exponential sum for $X \lesssim N^{1/d}$.
\begin{proposition} For all $\left\{x_1, \dots, x_N \right\} \subset \mathbb{T}^2$
$$ \left\| N -  \sum_{n=1}^{N}{e^{t\Delta}\delta_{x_n}} \right\|^2_{L^{\infty}(\mathbb{T}^d)} \gtrsim \sum_{ \|k\| \leq t^{-1/2} \atop k \neq 0}{ \left| \sum_{n=1}^{N}{ e^{2 \pi i \left\langle k, x_n \right\rangle}}\right|^2}.$$
\end{proposition}
Note that for $t \gtrsim N^{-2/d}$, we end up precisely with $\| k\| \lesssim t^{-1/2} \ll N^{1/d}$ and $k \neq 0$, which further motivates a precise understanding of the behavior of the exponential sum in this regime. It seems to suggest that lattices may be a good choice for $N=p^d$ but there might be too much loss in going from $L^{\infty}$ to $L^2$. Returning to the motivating problem of trying to measure temperature, it may be advantageous to not wait until $t$ units of time have passed and instead take some of the samples earlier than that.
\begin{quote}
\textbf{Open problem.} Are the two quantities
$$   \min_{x_n \in \mathbb{T}^d} \left\| 
N -  \sum_{n=1}^{N}{e^{t \Delta}\delta_{x_n}} \right\|_{L^{\infty}}  ~
\mbox{and} ~ \min_{x_n \in \mathbb{T}^d, t_n \leq t} \left\| 
N -  \sum_{n=1}^{N}{e^{t_n \Delta}\delta_{x_n}} \right\|_{L^{\infty}} $$
roughly comparable or can the second be substantially smaller than the first? Is there a difference if $L^{\infty}$ is
replaced by $L^{p}$?
\end{quote}
The question could be rephrased as whether in the approximation of constants by Gaussians it is advantageous to make them all as wide 
as allowed or whether one can gain significantly better approximations by taking them at varying width. The question seems nontrivial even for $d=1$.
 This curious approximation problem is likely to have other approximations as well (we refer to \cite{faul} for a related question in the theory of Gabor frames and Montgomery's work on theta functions \cite{montf}).

 \subsection{Pair correlation} If $(x_n)_{n=1}^{\infty}$ is a sequence on $[0,1]$, then we define the pair correlation function
 of the first $N$ points as
 $$ F_N(s) = \frac{1}{N} \left\{ 1\leq m,n \leq N: m \neq n \wedge \|x_m - x_n \| \leq \frac{s}{N} \right\}$$
 and $F(s) = \lim_{N \rightarrow \infty}{F_N(s)}$. Originally a concept in statistical mechanics, it has been of great interest in number theory in recent years \cite{pair1, pair2, pair3, pair4, pair5}. It is easy to see that for i.i.d. and uniformly distributed random variables, we have $F(s) = 2s$ almost surely. A recent result of Aistleitner, Lachmann \& Pausinger \cite{pair6} and, independently, Grepstad \& Larcher \cite{pair7} is that if a sequence satisfies $F(s) = 2s$, then it is uniformly distributed. A result of the author \cite{stein2}, inspired by these earlier works, is that if there exists a sequence of positive real numbers $(t_n)_{n=1}^{\infty}$ converging to 0 such that
 $$ \lim_{N \rightarrow \infty}{ \frac{1}{N^2} \sum_{i,j=1}^{N} \frac{1}{\sqrt{t_N}} \exp\left(-\frac{1}{\sqrt{t_N}} (x_i - x_j)^2\right) } = \sqrt{\pi},$$
 then $(x_n)_{n=1}^{\infty}$ is uniformly distributed. This is interesting when $t_N \sim N^{-2+2\varepsilon}$ because the values in the exponential function mostly depend on local gaps at scale $\sim N^{-1 + \varepsilon}$ and this is related to the notion of pair correlation and various generalizations.
 We observe that one usual criterion for uniform distribution of a sequence $(x_n)_{n=1}^{\infty}$ in $\mathbb{T}^d$ can be phrased as
 $$ \forall X \in \mathbb{N} \quad \lim_{N \rightarrow \infty}  \frac{1}{N^2} \sum_{k \in \mathbb{Z}^2 \atop \|k\| \leq X}{ \left| \sum_{n=1}^{N}{ e^{2 \pi i \left\langle k, x_n \right\rangle}}\right|^2} = 1.$$
 Returning to our upper bound
 $$  \frac{1}{N^2}  \sum_{k \in \mathbb{Z}^2 \atop \|k\| \leq X}{ \left| \sum_{n=1}^{N}{ e^{2 \pi i \left\langle k, x_n \right\rangle}}\right|^2} \leq \frac{1}{N^2} \sum_{i,j=1}^{N}{ X^d e^{-c_d X^2\|x_i -x_j\|^2}},$$
 we see that this gives rise to the natural analogue in higher dimensions. In particular, this would allow for the construction of criteria for uniform distribution analogous to \cite{stein2} in higher dimensions (where it might be more advantageous to work with the Jacobi theta function as opposed to Gaussians, we refer to \cite{stein2} for technical details). This would allow to obtain criteria for uniform distribution that only depend on the distribution of pairwise distances.

\section{Proof of Theorem 1 and Theorem 3}
\subsection{The Idea.}  
The projection of $f\in L^2(\mathbb{T})$ onto the $(2N+1)-$dimensional subspace generated by the first few exponentials is
easy to write down explicitly as a convolution with the Dirichlet kernel
$$ \pi_N f = D_N * f \qquad \mbox{where} \qquad D_N = \sum_{k=-N}^{N}{e^{2\pi i k x}} = \frac{\sin{((2N+1) \pi x)}}{\sin{(\pi x)}}$$
is the $N-$th Dirichlet kernel. Its slow decay implies that $D_N*f$ is a fairly nonlocal average of $f$ and its behavior is nontrivial. Averaging these projection generates better-behaved \textit{Fej\'{e}r kernel}
$$ F_N f := \frac{1}{N+1}\sum_{k=0}^{N}{\pi_n f} = \left(  \frac{1}{N+1}\sum_{k=0}^{N}{D_k}\right) * f = \left[\frac{1}{N+1}\left( \frac{\sin{((N+1)\pi x)}}{\sin{(\pi x)}}\right)^2\right] * f.$$ 
The results in this paper suggest that it may sometimes be advantageous to work with yet another averaging and consider smoothed Fej\'{e}r kernels (a C\'{e}saro mean of \textit{third} order applied to exponentials, of second order applied to the Dirichlet kernel and a simple averaging over Fej\'{e}r kernel). The crucial advantage is that this additional smoothing produces nonvanishing kernels.
\begin{lemma}[Averaging Fej\'{e}r Kernels] We have
$$  \frac{1}{N+1}\sum_{k=0}^{N}{F_k}(x)  \gtrsim \frac{N \log{\left(e + N|x|\right)}}{1+  N^2 x^2}.$$
\end{lemma} 
\begin{proof} By symmetry, it suffices to show the result for $0 \leq x \leq 1/2$. If $x \leq k^{-1}$, then $F_k(x) \sim k$
and we see that therefore both sides of the inequality are of size $\sim N$ if $x \lesssim N^{-1}$. It remains to deal with the case $N^{-1} \leq x \leq 1/2$.
We start by rewriting the sum as
$$ \frac{1}{N+1} \sum_{k=0}^{N}{ F_{k-1}(x) }= \frac{1}{N+1}  \frac{1}{\sin{(\pi x)^2}} \sum_{k=0}^{N}{ \frac{ \left(\sin{(k \pi x)}\right)^2}{k}    } \gtrsim \frac{1}{N+1} \frac{1}{x^2}\sum_{k=0}^{N}{ \frac{ \left(\sin{(k \pi x)}\right)^2}{k}    }.$$
We proceed by splitting the set $\left\{0, \dots, N\right\}$ into $\sim Nx$ blocks of size $\sim 1/x$. Summing over one such
block has the effect of serving as an approximation of the integral over $(\sin{k \pi x})^2$ over one period and thus
\begin{align*}
\frac{1}{N+1} \frac{1}{x^2} \sum_{k=0}^{N}{ \frac{ \left(\sin{(k \pi x)}\right)^2}{k} } &\sim \frac{1}{N+1} \frac{1}{x^2} \sum_{k=1}^{Nx} \sum_{\ell=1}^{1/x}  \frac{ \left(\sin{((k/x+\ell) \pi x)}\right)^2}{(k/x + \ell)} \\
&\gtrsim  \frac{1}{N+1} \frac{1}{x^2} \sum_{k=1}^{Nx} \frac{1}{(k+1)/x} \sum_{\ell=1}^{1/x}   \left(\sin{((k/x+\ell) \pi x)}\right)^2 \\
&\gtrsim \frac{1}{N+1} \frac{1}{x^2} \sum_{k=1}^{Nx} \frac{1}{(k+1)/x} \frac{1}{x} \gtrsim \frac{\log{(N x)}}{N x^2}.
\end{align*}
\end{proof}
The proof could be slightly sharpened in the regime $1/N \ll x \ll 1$ so as to yield sharp constants. In particular, for typical irrational values of $x$ in this regime we would expect that for $N$ large
$$ \frac{1}{N x^2} \sum_{k=0}^{N}{ \frac{ \left(\sin{(k \pi x)}\right)^2}{k} } \sim \frac{1}{2N x^2}\sum_{k=1}^{N}{\frac{1}{k}}.$$

\subsection{The proof.}
\begin{proof}[Proof of Theorem 1] We give the proof in two dimensions, the general case follows in the same manner (different coordinates decouple). Our goal is to bound
 $$\sum_{|k_1| \leq X_1 \atop |k_2| \leq X_2}{ \left| \sum_{n=1}^{N}{ e^{2 \pi i \left\langle k, x_n \right\rangle}}\right|^2}  \qquad \mbox{from below}.$$
We start by emulating the proof of Montgomery \cite{mont}. Clearly,
\begin{align*}
\sum_{|k_1| \leq X_1 \atop |k_2| \leq X_2}{ \left| \sum_{n=1}^{N}{ e^{2 \pi i \left\langle k, x_n \right\rangle}}\right|^2}  &\geq 
 \sum_{|k_1| \leq X_1 \atop |k_2| \leq X_2}{ \left(1 - \frac{|k_1|}{X_1}\right)   \left(1 - \frac{|k_2|}{X_2}\right) \left| \sum_{n=1}^{N}{ e^{2 \pi i \left\langle k, x_n \right\rangle}}\right|^2} \\
 &=   \sum_{m,n=1}^{N}  \sum_{|k_1| \leq X_1 \atop |k_2| \leq X_2}  \left(1 - \frac{|k_1|}{X_1}\right)   \left(1 - \frac{|k_2|}{X_2}\right)  e^{2 \pi i \left\langle k, x_m-x_n \right\rangle}.
 \end{align*}
 Writing $x_m = (x_{m,1}, x_{m,2})$ allows us to write the inner sum as 
 \begin{align*}
  &= \left(   \sum_{|k_1| \leq X_1 }  \left(1 - \frac{|k_1|}{X_1}\right)    e^{2 \pi i \left\langle k_1, x_{m,1} -x_{n,1} \right\rangle}\right) \left(   \sum_{|k_2| \leq X_2 }  \left(1 - \frac{|k_2|}{X_2}\right)    e^{2 \pi i \left\langle k_2, x_{m,2} - x_{n,2} \right\rangle}\right) \\
  &= F_{X_1-1}(x_{m,1} - x_{n,1})F_{X_2-1}(x_{m,2} - x_{n,2}).
  \end{align*}
   The proof of the Montgomery's Lemma concludes by using the nonnegativity of the Fej\'{e}r kernel and considering only the diagonal elements $m=n$. We proceed by noting that we can add another layer of averaging and argue
\begin{align*}
\sum_{|k_1| \leq X_1 \atop |k_2| \leq X_2}{ \left| \sum_{n=1}^{N}{ e^{2 \pi i \left\langle k, x_n \right\rangle}}\right|^2} &\geq\frac{1}{X_1}\sum_{s =1}^{X_1} \frac{1}{X_2} \sum_{t=1}^{X_2} \sum_{|k_1| \leq s \atop |k_2| \leq t}{ \left| \sum_{i=1}^{N}{ e^{2 \pi i \left\langle k, x_i \right\rangle}}\right|^2} \\
&= \frac{1}{X_1}\sum_{s =1}^{X_1} \frac{1}{X_2} \sum_{t=1}^{X_2} \sum_{m,n=1}^{N}  F_{s-1}(x_{m,1} - x_{n,1})F_{t-1}(x_{m,2} - x_{n,2}) \\
&= \sum_{m,n=1}^{N} \left( \frac{1}{X_1}\sum_{s =1}^{X_1}  F_{s-1}(x_{m,1} - x_{n,1}) \right)  \left( \frac{1}{X_2}\sum_{t =1}^{X_2}  F_{t-1}(x_{m,2} - x_{n,2}) \right).
\end{align*}
Lemma 1 implies that
$$ \frac{1}{X_1}\sum_{s =1}^{X_1}  F_{s-1}(x_{m,1} - x_{n,1}) \gtrsim \frac{X_1 \log{(e + X_1(x_{m,1} - x_{n,1}))}}{1 +
X_1^2 (x_{m,1} - x_{n,1})^2}$$
and likewise for the second term. This then implies Theorem 3. Specializing to the case $X_1=X_2=X$ and
ignoring the logarithm, we obtain
     $$\sum_{\|k\| \leq X}{ \left| \sum_{n=1}^{N}{ e^{2 \pi i \left\langle k, x_n \right\rangle}}\right|^2}
     \gtrsim \sum_{i,j =1}^{N}  \frac{X}{1+ X^2 (x_{i,1}-x_{j,1})^2}  \frac{X}{1+ X^2 (x_{i,2}-x_{j,2})^2}$$
and simplify the denominator
\begin{align*}  (1+ X^2 (x_{i,1}-x_{j,1})^2)(1+ X^2 (x_{i,2}- x_{j,2})^2) &= 1 + X^2\|x_{i} - x_{j}\|^2\\
& + X^4  (x_{i,1}- x_{j,1})^2 (x_{i,2}-x_{j,2})^2.
\end{align*}
Finally, we estimate 
\begin{align*}
 X^4  (x_{i,1}-x_{j,1})^2 (x_{i,2}-x_{j,2})^2 &\leq X^4( (x_{i,1}-x_{j,1})^4 +  (x_{i,2}-x_{j,2})^4) \\
 &\leq  X^4( (x_{i,1}-x_{j,1})^2 +  (x_{i,2}-x_{j,2})^2)^2 \\
 &\leq X^4 \| x_i - x_j\|^4
 \end{align*}
 and
\begin{align*}
1 + X^2\|x_i - x_j\|^2 + X^4 \|x_i - x_j\|^4 &\leq (1+X^2\|x_i -x_j\|^2)^2 \\
&\leq 2 + 2X^4 \|x_i - x_j\|^4,
\end{align*}
which implies Theorem 1.
\end{proof}

\section{Proof of Theorem 2}
\subsection{The Idea.}
The argument uses induction on scales and the inequality
$$ \sum_{i,j=1 \atop i \neq j}^{N}{ \frac{1}{ \|x_i -x_j\|^2}} \gtrsim N^2 \log{N}.$$
Given a point set, we distinguish two cases
\begin{itemize}
\item \textit{(no clusters)}: there are many $(\geq N/10)$ points with the property that their nearest neighbor is distance
at least $1/(300 \sqrt{N})$ or
\item \textit{(clusters)}: many $(\geq 9N/10)$ points have their nearest neighbor very close at a distance
of less than $1/(300 \sqrt{N})$.
\end{itemize}
The first case is simple: we can take the subset of points and notice that they are not very sensitive to changes
of the kernel in the origin, in particular they do not notice the absence of the singularity and we can apply the
existing result for the Riesz energy. The clustered case is more interesting: we will construct a new set 
of $3N/5$ points with the property that these $3N/5$ points are located in pairs of two in at most $3N/10$ different locations.
The bilinear nature of the problem implies that the energy of these $3N/5$ points is 4 times the energy of 
$3N/10$ points placed at their locations. We then iterate the procedure and note that since $4 > 10/3$, 
every application of the second step leads to superlinear growth, which implies the result.
We require one elementary Lemma of a potential-theoretic nature; its purpose is to show that the force exerted by two
nearby  sources is essentially comparable to the force exerted by two points located at their center of mass.
\begin{lemma} Let $a,b \in \mathbb{T}^2$ satisfy $\|a-b\| \leq 1/(100\sqrt{N})$. For all $c \in \mathbb{T}^2$, $N \geq 1$
$$ \left| \frac{1}{1 + N\|a-c\|^2} + \frac{1}{1 + N\|b-c\|^2} - \frac{2}{1 + N\left\|\frac{a+b}{2}-c\right\|^2} \right| \leq \frac{1}{5000}\frac{1}{1 + N\|a-c\|^2}$$
\end{lemma}
 The inequality is elementary and left to the reader. It would be possible to prove a stronger result since there is stronger decay in the regime $\|c-a\| \gg \|a-b\|$, however, this is not required.

\subsection{The proof.}
\begin{proof}[Proof of Theorem 2] Let $\left\{x_1, \dots, x_N \right\} \subset \mathbb{T}^2$ be given. We
start by noting
that
$$ N +  \sum_{i,j=1 \atop i \neq j}^{N}{ \frac{1}{1+ N \|x_i -x_j\|^2}} =   \sum_{i,j=1}^{N}{ \frac{1}{1+ N \|x_i -x_j\|^2}}$$
and we may thus work with the full term that includes self-interactions.
If there are many points that do not
have any other points in their immediate vicinity
$$ A= \left\{1 \leq i \leq N: \forall j \neq i:~ \|x_i - x_j \| \geq \frac{1}{300\sqrt{N}} \right\}, \quad \# A \geq \frac{N}{10},$$
then it is easy to conclude the result: we can simply select this subset of points and argue that, since all pairwise distances are bounded away from the singularity and thus comparable to the singular kernel
\begin{align*}
  \sum_{i,j=1}^{N}{ \frac{1}{1+ N \|x_i -x_j\|^2}} &\geq \sum_{i, j \in A}^{}{ \frac{1}{1 + N \|x_i -x_j\|^2}} = \#A +  \sum_{i, j \in A \atop i \neq j}^{}{ \frac{1}{1 + N \|x_i -x_j\|^2}}\\
  &\geq  \# A + \frac{1}{90000}   \sum_{i, j \in A \atop i \neq j}^{}{ \frac{1}{N \|x_i -x_j\|^2}} \\
&\geq \# A + \frac{1}{900000}   \sum_{i, j \in A \atop i \neq j}^{}{ \frac{1}{\# A \|x_i -x_j\|^2}}  \\
  &\gtrsim  \# A \log{(\# A)} \gtrsim N \log{N}.
  \end{align*}
We now deal with the more interesting remaining case 
$$ \# \left\{1 \leq i \leq N: \exists j \neq i \quad \|x_i - x_j \| \leq \frac{1}{300\sqrt{n}} \right\} \geq \frac{9N}{10}.$$
The next step is the construction of a subset $A \subset \left\{1, \dots, N\right\}$ of size at least $\#A \geq 3N/5$ and a bijective,
fixed-point-free map
$\pi: A \rightarrow A$ satisfying $\pi^2= \mbox{id}$ (i.e. grouping in pairs) such that
$$ \forall~a \in A \qquad \|x_{a} - x_{\pi(a)}\| \leq \frac{1}{100\sqrt{N}}.$$
The construction of this subset $A$ and map $\pi$ is done in the following explicit way (which, casually, can be summarized as follows: find your closest friend; if your friend is already matched up, try to see whether your close friend or the person who matched with your friend have any other friends that are still not matched and match with them).
\begin{enumerate}
\item We order the points in some arbitrary way $x_1, x_2, \dots, x_N$. Going from $1 \leq i \leq N$, 
if $\pi^{-1}(i)$ is not defined and if there exists a point $x_j$ 
with $\|x_i - x_j\| \leq 1/(300 \sqrt{N})$ and $\pi(j)$ undefined, then $\pi(i) := j$ and $\pi(j) :=i$.
\item After having done this, go through the list once more from the beginning. If $\pi(i)$ is undefined,
we check whether the nearest neighbor of $x_i$, $x_j$, satisfies $\|x_i - x_j\| \leq 1/(300\sqrt{N})$. If it
does, then $\pi(j)$ has to be defined and $\pi(j) \neq i$.
We then check whether there is a point $x_k$ in the $1/(300 \sqrt{N})$ neighborhood of $\left\{x_j, x_{\pi(j)} \right\}$ for which $\pi$ is not defined and, if so, define $\pi(i):=k, \pi(k) :=i$.
\end{enumerate}
It is easy to see that the algorithm is well defined on a subset $A \subset \left\{1,2,\dots, N\right\}$ and yields a bijective, fixed-point-free involution. It remains to show that $\#A \geq 3N/5$.\\

\textbf{Claim.} \textit{ $\pi$ that is defined on at least $3N/5$ of all points.}
\begin{proof}[Proof of Claim] By assumption, at least $9N/10$ have their nearest neighbor at distance $\leq 1/(300\sqrt{N})$. Let us now take one of these points $x_i$ and assume that $\pi(i)$ is undefined. This means that $x_i$ is
the only point in the $ 1/(300\sqrt{N})$-neighborhood of  $\left\{x_j, x_{\pi(j)} \right\}$, where $x_j$ is the nearest neighbor to $x_i$. Therefore, for every point $x_i$ for which $\pi$ is not defined, we can find two unique points $\left\{x_j, x_{\pi(j)} \right\}$ for which $\pi$ is defined. This implies that $\pi$ can at most be undefined for one third of the $9N/10$ points with a close
nearest neighbor, which implies the statement.
\end{proof}
We now only concentrate on these points and set
$$ A = \left\{1 \leq i \leq N:~\pi~\mbox{is defined} \right\},$$
 focus on that selected subset and use the trivial bound
$$   \sum_{i,j=1 \atop i \neq j}^{N}{ \frac{1}{1+ N \|x_i -x_j\|^2}} \geq   \sum_{i, j \in A \atop i \neq j}^{}{ \frac{1}{1+ N \|x_i -x_j\|^2}}.$$
The next step is the construction of a new point set based on $\left\{x_i: i \in A\right\}$ by replacing every pair
$(x_i, x_{\pi(i)})$ by two points located in $(x_i + x_{\pi(i)})/2$. We denote this new set of points by $\left\{y_1, y_2, \dots, y_{\#A} \right\} \subset \mathbb{T}^2$.
It remains to bound the effect that this is having on the energy. Since points are already grouped in pairs of 2, we will compute the effect on the energy by seeing how it affects the interactions between two pairs of two points each. Let us assume the pairs are $a_1, a_2$ (having their geometric average at $a_3$) and $b_1, b_2$ (having their geometric average at $b_3$). Abbreviating
$$ E(x, y) = \frac{1}{1+N\|x-y\|^2},$$
we see from Lemma 2 that
\begin{align*}
\left| E(a_1, b_1) + E(a_2, b_1) - 2E(a_3, b_1)\right| &\leq \frac{1}{10000} \left( E(a_1, b_1) + E(a_2, b_1)\right) \\
\left| E(a_1, b_2) + E(a_2, b_2) - 2E(a_3, b_2)\right| &\leq \frac{1}{10000} \left( E(a_1, b_2) + E(a_2, b_2)\right)
\end{align*}
Moreover, by monotonicity, $E(a_1, a_3) \geq E(a_1, a_2)$.

Summing up, we see that replacing $a_1, a_2$ by two points in $a_3$ changes the energy by at most $1/10000$ of the original energy. In the next step, we see that
\begin{align*}
\left| E( b_1, a_3) + E(b_2, a_3) - 2E(b_3, a_3)\right| \leq \frac{1}{5000} \left( E(b_1, a_3) + E(b_2, a_3)\right)
\end{align*}
which means that replacing $b_1$ and $b_2$ by their geometric average also has only a small effect.
\begin{center}
\begin{figure}[h!]
\begin{tikzpicture}[scale=1.8]
\filldraw (0,0) circle (0.04cm);
\filldraw (0.5,0.3) circle (0.04cm);
\filldraw (0.5,0.5) circle (0.04cm);
\filldraw (1,1) circle (0.04cm);
\filldraw (1,0) circle (0.04cm);
\filldraw (0.7,0) circle (0.04cm);
\filldraw (1.4,1) circle (0.04cm);
\filldraw (1.6,0.8) circle (0.04cm);
\filldraw (1.5,0.5) circle (0.04cm);
\draw  (01,1) -- (01.4,1);
\draw  (0.5,0.3) -- (0.5,0.5);
\draw  (1.5,0.5) -- (1.6,0.8);
\draw  (0.5,0.3) -- (0.5,0.5);
\draw  (1.0,0) -- (0.7,0);
\draw (3,0) circle (0.04cm);
\draw (4,1) circle (0.04cm);
\draw (3.5,0.3) circle (0.04cm);
\draw (3.5,0.5) circle (0.04cm);
\draw (4,0) circle (0.04cm);
\draw (3.7,0) circle (0.04cm);
\draw (4.4,1) circle (0.04cm);
\draw (4.6,0.8) circle (0.04cm);
\draw (4.5,0.5) circle (0.04cm);
\draw  (4,1) -- (4.4,1);
\filldraw (4.2,1) circle (0.04cm);
\draw  (3.5,0.3) -- (3.5,0.5);
\filldraw (3.5,0.4) circle (0.04cm);
\draw  (4.5,0.5) -- (4.6,0.8);
\filldraw (4.55,0.65) circle (0.04cm);
\draw  (3.5,0.3) -- (3.5,0.5);
\filldraw (3.5,0.4) circle (0.04cm);
\draw  (4.0,0) -- (3.7,0);
\filldraw (3.85,0) circle (0.04cm);
\end{tikzpicture}
\caption{After having distilled many pairs, we replace them by their geometric average.}
\end{figure}
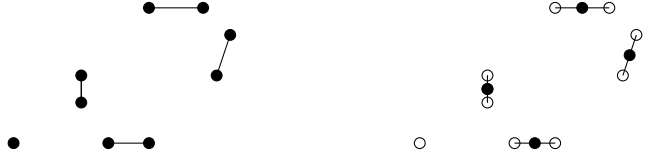
\end{center}
\vspace{0pt}
 Altogether, by repeating this
procedure over all pairs of pairs of two points, 
$$ \left|   \sum_{i, j \in A}^{}{ \frac{1}{1+ N \|x_i -x_j\|^2}} -   \sum_{i, j \leq \# A}^{}{ \frac{1}{1+ N \|y_i - y_j\|^2}} \right| \leq \frac{1}{1000}   \sum_{i, j \in A}^{}{ \frac{1}{1+ N \|x_i -x_j\|^2}}$$
and therefore
$$    \sum_{i, j \leq \#A}^{}{ \frac{1}{1+ N \|y_i - y_j\|^2}}  \geq \frac{999}{1000}   \sum_{i, j \in A}^{}{ \frac{1}{1+ N \|x_i -x_j\|^2}}.$$

Moreover, points in $\left\{y_1, y_2, \dots, y_{\#A} \right\} \subset \mathbb{T}^2$ come in pairs of two and by removing duplicates, we obtain a set $\left\{z_1, z_2, \dots, z_{\#A/2} \right\} \subset \mathbb{T}^2$ with
$$    \sum_{i, j \leq \# A}^{}{ \frac{1}{1+ N \|y_i - y_j\|^2}}  = 4    \sum_{i, j \leq \#A/2}^{}{ \frac{1}{1+ N \|z_i - z_j\|^2}}.$$  
Summarizing, we have obtained a set of at least $3N/10$ points ($30\%$ of the original points) that have at most $25(1000/999)\% \leq 26\%$ of the energy of the original set of points. This allows us to conclude the argument: given any point set $\left\{x_1, \dots, x_N \right\} \subset \mathbb{T}^2$, we iterate this procedure until, for the first time, we end up in the case of many points not having a close nearest neighbor. If this requires $k$ iterations, then we end up with a set of points of size at least $(0.3)^k n$ and energy at most $0.26^k E$, where $E$ is the energy of the original set. The energy $E_1$ of the reduced set is at least $(0.3)^k n \log{\left((0.3)^k n\right)}$ and thus
$$ \sum_{i,j=1}^{N}{ \frac{1}{1 + N\|x_i -x_j\|^2}} \geq \frac{1}{0.26^k} 0.3^k  N \log{\left(0.3^k N\right)}.$$
If $k \leq \log{N}/100$, then
$$  \frac{0.3^k  }{0.26^k} N \log{\left(0.3^k N\right)} \geq    N \log{\left( 0.3^{\frac{\log{N}}{100}} N\right)} \gtrsim N \log{\sqrt{N}} \sim N \log{N}.$$
Otherwise, we stop after $\log{N}/100$ iterations, use the number of remaining points as a trivial lower bound on the energy of the remaining point set and conclude
 $$ \sum_{i,j=1}^{N}{ \frac{1}{1 + N\|x_i -x_j\|^2}} \geq \frac{1}{0.26^k} [0.3^k  N] \gtrsim N^{1.001} \gtrsim N \log{N}.$$
\end{proof}
The proof has the following dynamical interpretation: if $\left\{x_1, \dots, x_N \right\} \subset \mathbb{T}^2$ satisfies
$$ \sum_{i,j=1}^{N}{ \frac{1}{1 + N\|x_i -x_j\|^2}} \leq  c N \log{N},$$
then the process of repeatedly `melting nearby pairs points into one' is bound to result in a well-separated point set after at most $ \log{c}$ steps. This can be understood as a regularity statement for optimal configurations. More precise information would be of interest; in particular, repeating the question from the introduction, is every optimal configuration maximally separated (meaning that we have a uniform separation $\|x_i - x_j\| \gtrsim N^{-1/2}$ whenever $i \neq j$)?

\section{Proof of Theorem 4 and Corollary 2}

\subsection{Proof of Theorem 4.}
\begin{proof} The main idea is to bound the exponential sum by a weighted exponential sum over the entire space with a weight chosen such that we have rapid decay as soon as $\|k\| \gtrsim X$ and, simultaneously, obtain meaningful quantities that can be interpreted in a different manner. The heat evolution provides a natural example and we can write
\begin{align*}
 \sum_{k \in \mathbb{Z}^d \atop \|k\| \leq X}{ \left| \sum_{n=1}^{N}{ e^{2 \pi i \left\langle k, x_n \right\rangle}}\right|^2}
&\leq e \sum_{k \in \mathbb{Z}^d}{ e^{-\|\xi\|^2 X^{-2}} \left| \sum_{n=1}^{N}{ e^{2 \pi i \left\langle k, x_n \right\rangle}}\right|^2} \\
&= e\left\| e^{X^{-2} \Delta} \left( \sum_{n=1}^{N}{ \delta_{x_n}}\right) \right\|_{L^2(\mathbb{T}^d)}^2,
\end{align*}
where, as in the introduction, $e^{t \Delta} f$ is the heat evolution at time $t$ with $f$ as initial datum. We will now bound this expression using information about the heat kernel. Squaring out, collecting diagonal and off-diagonal terms gives
\begin{align*}
 \left\|   \sum_{n=1}^{N}{e^{X^{-2} \Delta} \delta_{x_n}} \right\|_{L^2(\mathbb{T}^d)}^2 &=
 \sum_{n=1}^{N} \int_{\mathbb{T}^d}{ [e^{X^{-2}\Delta} \delta_{x_n}](x)^2 dx} \\
 &+  \sum_{i, j=1 \atop i \neq j}^{N} \int_{\mathbb{T}^d}{ [e^{X^{-2}\Delta} \delta_{x_i}](x)  [e^{X^{-2}\Delta} \delta_{x_j}](x) dx}.
\end{align*}
The simple heat-kernel estimate
 $$ [e^{t \Delta} \delta_{y}](x) \lesssim \frac{1}{t^{d/2}} \exp{\left(-\frac{c \|x-y\|^2}{ t}\right)},$$
 where $c$ is a constant depending only on the dimension, implies
$$ \int_{\mathbb{T}^d}{ [e^{X^{-2}\Delta} \delta_{x_n}](x)^2 dx} \lesssim X^{2d} \int_{\mathbb{T}^d}{ \exp{\left(- 2c X^2 \|x\|^2\right)} dx} \lesssim X^d.$$
The second integral can be reduced to the heat kernel estimate by using self-adjointness of the heat semigroup
 $$ \left\langle e^{t\Delta} \delta_x, e^{t \Delta} \delta _y\right\rangle =  \left\langle e^{2t\Delta} \delta_x,  \delta_y\right\rangle = [e^{2t\Delta} \delta_x](y)$$
 to obtain
 $$  \int_{\mathbb{T}^d}{ [e^{X^{-2}\Delta} \delta_{x_i}](x)  [e^{X^{-2}\Delta} \delta_{x_j}](x) dx} \lesssim X^d \exp\left(-2 c X^2 \|x_i - x_j\|^2\right)$$
 which, after summation, implies 
\begin{align*} \sum_{k \in \mathbb{Z}^d \atop \|k\| \leq X}{ \left| \sum_{n=1}^{N}{ e^{2 \pi i \left\langle k, x_n \right\rangle}}\right|^2} &\lesssim N X^d +   \sum_{i, j=1 \atop i \neq j}^{N}{ X^d e^{- c X^2 \|x_i - x_j\|^2}} \\
 &=  \sum_{i, j=1}^{N}{ X^d e^{- c X^2 \|x_i - x_j\|^2}}
 \end{align*}
\end{proof}

\subsection{Proof of Corollary 2}
Montgomery's Lemma implies the lower bound, it suffices to show that the upper bound has matching asymptotic behavior.
\begin{proposition} If the set $\left\{x_1, \dots, x_N \right\} \subset \mathbb{T}^d$ satisfies $\|x_i - x_j\| \gtrsim N^{-1/d}$ whenever $i \neq j$ and $X \gtrsim N^{1/d}$, then 
$$   \sum_{i, j=1}^{N}{ X^d e^{-c X^2 \|x_i - x_j\|^2}} \sim N X^d.$$
\end{proposition}
\begin{proof}The diagonal terms $i=j$ contribute $NX^d$. It remains to show that the off-diagonal terms do not contribute more.
 Fix an arbitrary $x_i$. Since the points are $N^{-1/d}-$separated, we have that
$$ \# \left\{ x_j:   \frac{k}{N^{1/d}} \leq \|x_i - x_j \| \leq \frac{k+1}{N^{1/d}}\right\} \lesssim k^{d-1}.$$
This suffices to conclude the result since 
\begin{align*}
 \sum_{i=1}^{N} \sum_{j \neq i}{  X^d e^{-cX^2 \|x_i - x_j\|^2}} &\lesssim X^d N \sum_{k=1}^{N^{1/d}}{ k^{d-1} e^{-c X^2 \frac{k^2}{N^{2/d}}}} \\
&\lesssim X^d N\left(e^{-c X^2 N^{-2/d}} +  \int_{1}^{N^{1/d}+1}{ k^{d-1} e^{- c X^2 N^{-2/d} k^2} dk}\right) \\
&\lesssim X^d N \left( 1+   \int_{1}^{N^{1/d}+1}{ k^{d-1} e^{- c X^2 N^{-2/d} k^2} dk}\right).
\end{align*}
Since $X^2 N^{-2/d} \gtrsim 1$, we can bound the integral by
$$   \int_{1}^{N^{1/d}+1}{ k^{d-1} e^{-c  X^2 N^{-2/d} k^2} dk} \lesssim   \int_{1}^{\infty}{ k^{d-1} e^{-c  k^2} dk} \lesssim_{c,d} 1.$$

\end{proof}

\subsection{Proof of the Propositions}
\begin{proof}
If $f \in L^1(\mathbb{T})$, then self-adjointness of the heat semigroup, $e^{t\Delta}1 = 1$ and the $L^1-L^{\infty}$ duality imply
$$ \left\langle e^{t\Delta}f, N - \sum_{n=1}^{N}{\delta_{x_n}}  \right\rangle = \left\langle f,  N - e^{t\Delta} \sum_{n=1}^{N}{\delta_{x_n}} \right\rangle \leq \|f\|_{L^1} \left\|  N - e^{t\Delta} \sum_{n=1}^{N}{\delta_{x_n}} \right\|_{L^{\infty}}.$$
Furthermore, by taking $f$ to be an approximation of a Dirac measure located where the maximum is assumed, we see that inequality cannot be improved.
The second proposition is equally simple since
\begin{align*} \left\| N -  \sum_{n=1}^{N}{e^{t\Delta}\delta_{x_n}} \right\|^2_{L^{\infty}(\mathbb{T}^d)} &\geq  \left\| N -  \sum_{n=1}^{N}{e^{t\Delta}\delta_{x_n}} \right\|^2_{L^{2}(\mathbb{T}^d)}\\
& = \sum_{k \in \mathbb{Z}^d \atop k \neq 0}{ e^{- k^2 t} \left| \sum_{n=1}^{N}{e^{2\pi i \left\langle k, x_n \right\rangle}} \right|^2}\\
&\gtrsim  \sum_{\| k\| \leq  t^{-1/2} \atop k \neq 0}{ \left| \sum_{n=1}^{N}{e^{2\pi i \left\langle k, x_n \right\rangle}} \right|^2},
\end{align*}
which gives the result.
The same simple argument gives
$$ \sup_{\|f\|_{L^p}\leq 1}\left| \left\langle e^{t\Delta}f, N - \sum_{n=1}^{N}{\delta_{x_n}}  \right\rangle\right| =  \left\|  N - e^{t\Delta} \sum_{n=1}^{N}{\delta_{x_n}} \right\|_{L^{q}}$$ 
whenever $1/p + 1/q=1$, which naturally motivates the $L^q-$version of the problem. 
\end{proof}

\end{document}